\documentclass[12pt]{article}
\usepackage{amsfonts}

\usepackage[latin1]{inputenc}
\usepackage{amsmath,amssymb}
\usepackage{latexsym}
\usepackage[active]{srcltx}
\usepackage[
bookmarks=true,         
bookmarksnumbered=true, 
colorlinks=true, pdfstartview=FitV, linkcolor=blue, citecolor=blue,
urlcolor=blue]{hyperref}

\newtheorem{theorem}{Theorem}[section]
\newtheorem{corollary}[theorem]{Corollary}
\newtheorem{lemma}[theorem]{Lemma}
\newtheorem{proposition}[theorem]{Proposition}

\newtheorem{remark}[theorem]{Remark}
\numberwithin{equation}{section}
\def\R{{\mathbb R}}
\def\E{{{\mathbb E}\,}}
\def\P{{\mathbb P}}

\def\Z{{\mathbb Z}}

\def\square{{\vcenter{\vbox{\hrule height.3pt
        \hbox{\vrule width.3pt height5pt \kern5pt
           \vrule width.3pt}
        \hrule height.3pt}}}}

\def\tlint{{- \kern-0.85em \int \kern-0.2em}}
\def\dlint{{- \kern-1.05em \int \kern-0.4em}}

\def\bN {{\mathbb N}}

\def \eref#1{\hbox{(\ref{#1})}}

\def \eref#1{\hbox{(\ref{#1})}}

\newenvironment{proof}[1][Proof]{\noindent\textit{#1.} }{\hfill \rule{0.5em}{0.5em}}

\begin{document}

\title{Markov Chain Approximations to Singular Stable-like Processes}

\date{\empty }
\author{Fangjun Xu\thanks{F. Xu is supported in part by the Robert Adams Fund.}\\
Department of Mathematics \\
University of Kansas \\
Lawrence, Kansas, 66045 USA}

\maketitle

\begin{abstract}
\noindent We consider the Markov chain approximations for a class of singular stable-like processes. First we obtain properties of some Markov chains. Then we construct the approximating Markov chains and give a necessary condition for the weak convergence of these chains to the singular stable-like processes.\vskip.2cm

 \vskip.2cm


\end{abstract}

\section{Introduction}
A class of singular stable-like processes $X$ is considered in \cite{xu}. These processes $X$
correspond to the Dirichlet forms
\begin{equation} \label{form}
\left\{
\begin{array}{ll}
\mathcal{E}(f,f)=\int_{\R^d}\int_{\R^d}\big(f(y)-f(x)\big)^2J(x,y)\,m(dy)\,dx,\\
\\
\mathcal{F}=\big\{f\in L^2(\R^d):\mathcal{E}(f,f)<\infty\big\},
\end{array}\right.
\end{equation}
where $m(dy)$ is the measure on the union of coordinate axes $\cup^d_{i=1}\R_i$ with $\R_i$ being the i-th coordinate axis of $\R^d$ and $m$ restricted to each $\R_i$ being one-dimensional Lebesgue measure on $\R$. The jump kernel $J(x,y)$
satisfies
\[
J(x,y)=\left\{
\begin{array}{rl}
\frac{c(x,y)}{|x-y|^{1+\alpha}}, &~\text{if}~ y-x\in
\cup^d_{i=1}\R_i\backslash\{0\};\\
0, &~\text{otherwise},
\end{array}\right.
\]
where $c(x,y)=c(y,x)$ and $0<\kappa_1\leq c(x,y)\leq
\kappa_2<\infty$ for all $x$ and $y$ in $\R^d$.

In this paper, we consider the Markov chain approximations for the processes $X$ in \cite{xu}. In the last few years, Markov chain approximations for symmetric Markov
processes have received a lot of attention. Stroock and Zheng proved the Markov chain approximations to symmetric diffusions in \cite{stroock_zheng}. We refer to \cite{bass_kumagai2}, \cite{husseini_kassmann} and \cite{chen_kim_kumagai} for Markov chain approximations to general symmetric Markov processes and  \cite{deuschel_kumagai} for Markov chain approximations to non-symmetric diffusions. The Markov chain approximations to $X$ are not considered in the above references since the processes $X$ have singular jump kernels. It is natural to ask whether $X$ could be approximated by Markov chains. If so, under what conditions, such approximation holds. The main difficulty is to get the near diagonal lower bounds  in Proposition \ref{lower}. We use weighted Poincar\'{e} inequalities to obtain these lower bounds. This technique first appeared in \cite{stroock_saloff}, see also \cite{saloff_coste} and \cite{xu1}.

The paper is organized as follows. In Section 2 we introduce
notation and define Markov chains related to $X$. In Section 3, we first construct a sequence of Markov chains. Then we obtain heat kernel estimates, exit time estimates and the regularity for these chains. In Section 4 we show the Markov chain approximations for processes $X$. Throughout this paper, if not mentioned otherwise, the letter $c$ with or without a subscript denotes a positive finite constant whose exact value is unimportant and may change from line to line.

\medskip

\section{Preliminaries}

Let $C(\cdot,\cdot):\Z^d\times\Z^d\to [0,\infty)$ be the function satisfying 
\begin{itemize}
\item[(a)] $C(x,y)=C(y,x)$ for all $x,y\in\Z^d$;
\item[(b)] There exist positive constants
$\kappa_1$ and $\kappa_2$ such that
\[
\left\{
\begin{array}{cl}
\frac{\kappa_1}{|y-x|^{1+\alpha}}\leq C(x,y)\leq
\frac{\kappa_2}{|y-x|^{1+\alpha}}, &\qquad\text{if}~y-x\in\cup^d_{i=1}\Z_i\backslash\{0\},\\
C(x,y)=0, &\qquad\text{otherwise},
\end{array}\right.
\]
where $\Z_i=\Z e_i$ with $e_i$ being the i-th vector in $\R^d$.
\end{itemize}
For any $x$ and $y$ in $\Z$, $C(x,y)$ is called the conductance between $x$ and $y$. Set
\[
G_x:=\sum_{y\in\Z^d}C(x,y)=\sum_{z\in \cup^d_{i=1}\Z_i}C(x,x+z).
\]
We define a symmetric Markov chain $\widetilde{Y}$ on
$\Z^d$ by
\[
\P(\widetilde{Y}_1=y\,|\,\widetilde{Y}_0=x)=\frac{C(x,y)}{G_x},\quad
\text{for}\; x,y\in\Z^d.
\]
The Markov chain $\widetilde{Y}$ is discrete in time and in space.
We next introduce the continuous time version of $\widetilde{Y}$. Let
$Y$ be a process that waits at a point in $\Z^d$ for a length of
time that is exponential with parameter 1, then jumps according to
the jump probabilities of $\widetilde{Y}$. After that, the process
$Y$ waits at the new point for a length of time that is exponential
with parameter 1 and independent of what has gone before, and so on.
The process $Y$ defined above is the continuous time version of
$\widetilde{Y}$. The continuous time and continuous state process
closely related to both $\widetilde{Y}_n$ and $Y_t$ is the process
$X$ corresponding to the Dirichlet form $(\mathcal{E},\mathcal{F})$ in \eref{form}.

Let $q_Y(t,x,y)$ be the transition density of $Y$. Since the conductance function of $Y$ satisfies conditions (A1)-(A4) in \cite{xu1}, Proposition 2.2 in \cite{xu1} implies the following result. 
\begin{proposition} \label{bound1} For all $x$ and $y$ in $\Z^d$, there exists a positive constant $c_1$ such that
\[
q_Y(t,x,y)\leq c_1(t^{-d/\alpha}\wedge 1),\quad \text{for all}\; t>0.
\]
\end{proposition}

\medskip

\section{Heat Kernel Estimates and Regularity}
In this section, we first define a sequence of Markov chains from $Y$. Then we obtain heat kernel estimates, exit time estimates and the regularity result for these chains.

For each $\rho\geq 1$, set $\mathcal{S}=\rho^{-1}\mathbb{Z}^d$. For each $x\in\mathcal{S}$ and $A\subset\mathcal{S}$, let $\mu^{\rho}_x=\rho^{-d}$
and $\mu^{\rho}(A)=\sum\limits_{y\in A}\mu^{\rho}_y$. Define the rescaled process $V$ as 
\[
V_t=\rho^{-1}Y_{\rho^{\alpha}t},\quad\text{for}\; t\geq 0.
\] 
We see that the Dirichlet form corresponding to $V$ is
\[
\left\{
\begin{array}{rl}
\mathcal{E}^{\rho}(f,f)= \sum\limits_{\mathcal{S}}\sum\limits_{\mathcal{S}}\big(f(y)-f(x)\big)^2C^{\rho}(x,y),\\
\mathcal{F}_{\rho}= \big\{f\in L^2(\mathcal{S},\mu^{\rho}):\;\mathcal{E}^{\rho}(f,f)<\infty\big\},
\end{array}\right.
\]
where $C^{\rho}(x,y)=\rho^{\alpha-d}C(\rho x,\rho y)$ for all $x,y\in\mathcal{S}$.

Write $p(t,x,y)$ for the transition density of $V$. Then
\begin{equation} \label{scale}
p(t,x,y)=\rho^{d}q_Y(\rho^{\alpha}t,\rho x,\rho y)
\end{equation}
for all $x,y\in\Z^d$ and $t>0$.
\begin{proposition} For all $\rho\geq 1$, there exists $c_1$ such that
\[
p(t,x,y)\leq c_1t^{-d/\alpha}.
\]
\end{proposition}
\begin{proof} This follows from \eref{scale} and Proposition \ref{bound1}.
\end{proof}

\medskip

For each $\lambda\geq 1$, let $V^\lambda$ be the process $V$ with
jumps greater than $\lambda$ removed. Write $p^{\lambda}(t,x,y)$ for
the transition density of the truncated process $V^{\lambda}$. The argument in the proof of Lemma 2.5 in \cite{xu1} gives the following off-diagonal upper bound for $p^{\lambda}(t,x,y)$.
\begin{lemma} \label{bound2} For all $t>0$ and $x,y\in\mathcal{S}$, there exist $c_1$ and $c_2$
such that
\[
p^{\lambda}(t,x,y)\leq c_1t^{-d/\alpha}e^{c_2t-|x-y|/\lambda}.
\]
\end{lemma}

For any set $A\subset\mathcal{S}$, define
\begin{equation*}
T_A(V)=\inf\big\{t\geq 0: V_t\notin A\big\}\quad\text{and}\quad
\tau_A(V)=\inf\big\{t\geq 0: V_t\in A\big\}.
\end{equation*}
The upper bound in Lemma \ref{bound2} implies the following exiting time
estimates for $V$, whose proof can be found in Proposition 3.4 of 
\cite{bass_kumagai1} and Proposition 4.1 of \cite{chen_kumagai}.
\begin{theorem} \label{exit1}
For $a>0$ and $0<b<1$, there exists $\gamma=\gamma(a,b)\in(0,1)$
such that for any $R\geq 1$ and $x\in\mathcal{S}$,
\[
\P^x\big(\tau_{(x,aR)}(V)<\gamma R^\alpha\big)\leq b.
\]
\end{theorem}

\medskip

Recall the definition of the rescaled process $V$. Using Proposition 2.7, Remark 2.8 and Theorem 2.11 in \cite{xu1}, we obtain the following near diagonal lower bound for $p(t,x,y)$.
\begin{proposition} \label{lower}
There exists $c>0$ such that
\[
p(t,x,y)\geq c t^{-d/\alpha},
\]
for all $t\geq \rho^{-\alpha}$ and 
$|x-y|<2t^{1/\alpha}$.
\end{proposition}

Theorem \ref{exit1} and the proof of Lemma 4.5 in
\cite{bass_kumagai1} imply the following lemma.
\begin{lemma} \label{hit}
Given $\delta>0$ there exists $\kappa>0$ such that if
$x,y\in\mathcal{S}$, and $A\subset\mathcal{S}$ with $dist(x,A)$ and
$dist(y,A)$ both larger than $\kappa t^{1/2}$, then
\[
\P^x(V_t=y, T_A\leq t)\leq \delta t^{-d/\alpha}\rho^{-d}.
\]
\end{lemma}

\begin{proposition}  \label{exit2} For all $t\geq\rho^{-\alpha}$, there exist $c_1>0$ and $\theta\in(0,1)$ such that if $|x-z|$,
$|y-z|\leq t^{1/\alpha}$, $x,y,z\in\mathcal{S}$,
 and $r\geq t^{1/\alpha}/\theta$, then
\begin{equation} \label{eq1}
\P^x(V_t=y,\tau_{B_{(z,r)}}>t)\geq c_1 t^{-d/\alpha}\rho^{-d}.
\end{equation}
\end{proposition}
\begin{proof} This follows easily from Proposition \ref{lower} and Lemma \ref{hit}. 
\end{proof}
\begin{remark} 
The above proposition still holds if we replace `` $|x-z|$, $|y-z|\leq t^{1/\alpha}$, $x,y,z\in\mathcal{S}$ '' with `` $|x-y|\leq 2t^{1/\alpha}$, $x,y\in\mathcal{S}$ '' and `` z '' in \eref{eq1} with `` x '', respectively.
\end{remark}

As an application of Proposition \ref{exit2}, we have
\begin{corollary} \label{exit3}
For each $0<\epsilon<1$, there exists
$\theta=\theta(\epsilon)\in(0,1)$ with the following property: if
$x,y\in\mathcal{S}$ with $|x-y|\leq t^{1/\alpha}$, $t\in[0,\theta^{\alpha}r^\alpha)$, and $\Gamma\subset B(y,
t^{1/\alpha})\cap\mathcal{S}$ satisfies
$\mu^{\rho}(\Gamma)t^{-d/\alpha}\geq\epsilon$, then
\[
\P^x(V_t\in\Gamma~and~\tau_{B(y,r)}>t)>c_1\epsilon.
\]
\end{corollary}

\medskip

In the remaining of this section we show the regularity result for $V$. Since $V$ is a Hunt process, there is a L\'{e}vy system formula for it. We refer to \cite{chen_kumagai} for its proof.
\begin{lemma} \label{levy}
Let $f:\R^+\times \mathcal{S}\times\mathcal{S}\to\R^+$ be a bounded
measurable function vanishing on the diagonal. Then, for all
$x\in\mathcal{S}$ and predictable stopping time $T$, we have
\[
\E^x\big[\sum_{s\leq
T}f(s,V_{s-},V_s)\big]=\E^x\Big[\int^T_0 \sum_{y\in\mathcal{S}}f(s,V_s,y)\, C^{\rho}(V_s,y)\rho^{d}\, ds\Big].
\]
\end{lemma}

Let $W_t=W_0+t$ be a deterministic process. Then $Z=(W_t, V_t)$ is the
space-time process on $\R^+\times\mathcal{S}$ associated with $V$.
We say that a nonnegative Borel measurable function $h(t,x)$ on
$\R^+\times\mathcal{S}$ is parabolic in an open set
$B\subset\R^+\times\mathcal{S}$ if for all open relative compact
sets $B'\subset B$ and $(t,x)\in B'$,
\[
q(t,x)=\E^{(t,x)}\big[h(Z_{\tau(B'; Z_s)})\big].
\]

For any $t_0>0$, by Lemma 4.5 in \cite{chen_kumagai}, the function
$q^{\rho}(t,x)=p(t_0-t,x,y)$ is parabolic in $[0,t_0)\times\mathcal{S}$.

\begin{lemma} \label{lem1}
For each $\delta\in(0,1)$, there exists
$\gamma=\gamma_{\delta}\in(0,1)$ such that for $t>0$, and
$x\in\mathcal{S}$, if $A\subset
Q^{\rho}_{\gamma}(t,x,r):=[t,t+\gamma
r^{\alpha}]\times (B(x,r)\cap \mathcal{S})$ satisfies
$\frac{m\otimes\mu^{\rho}(A)}{m\otimes\mu^{\rho}(Q^{\rho}_{\gamma}(t,x,r))}\geq\delta$,
then
\[
\P^{(t,x)}\big( T_A(Z)<\tau_{Q^{\rho}_{\gamma}(t,x,r)}(Z)\big)\geq c_1\delta.
\]
\end{lemma}
\begin{proof} Thanks to Corollary \ref{exit3}, this follows from using similar arguments in the proof of Lemma 4.7 in \cite{bass_kumagai1}.
\end{proof}

\begin{lemma} \label{lem2}
There exists a positive constant $c_1$ such that for $s>2r$ and
$(t,x)\in[0,\infty)\times\mathcal{S}$
\begin{equation*}
\P^{(t,x)}\big(Z_{\tau_{Q^{\rho}(t,x,r)}}\notin
Q^{\rho}(t,x,s)\big) \leq c_1\frac{r^{\alpha}}{s^{\alpha}}.
\end{equation*}
\end{lemma}
\begin{proof} For simplicity of notation, we write
$\tau$ for $\tau_{Q^{\rho}(t,x,r)}$. Note that 
\begin{align*}
\P^{(t,x)}\big(Z_{\tau}\notin
Q^{\rho}(t,x,s)\big)=\P^{x}\big(V_{\tau}\notin B(x,s)\cap\mathcal{S}
;\tau\leq\gamma r^{\alpha}\big).
\end{align*}
By Lemma \ref{levy},
\begin{equation*}
\P^{x}\big(V_{\tau}\notin B(x,s)\cap\mathcal{S}\big)
=\E^x\big[\int^{\tau}_0 \sum_{|y-x|\geq
s}C^{\rho}(V_{t},y)\, \rho^{d}\, dt\big]\leq
c_2s^{-\alpha}\E^x(\tau).
\end{equation*}
On the other hand,
\begin{equation*}
1\geq\P^{x}\big(V_{\tau}\notin B(x,r)\cap\mathcal{S})
=\E^x\big[\int^{\tau}_0 \sum_{|y-x|\geq
r}C^{\rho}(V_{t},y)\,\rho^{d}\, dt\big]\geq
c_3r^{-\alpha}\E^x(\tau).
\end{equation*}
Combining these estimates gives the required inequality.
\end{proof}

\medskip

We next derive the regularity result for $V$, which is also needed in the Markov chain approximations.

\begin{theorem} \label{regular}
There exist $c>0$ and $\beta>0$ (independent of $R$ and $\rho$) such
that for every bounded parabolic function $q$ in
$Q^{\rho}(0,x_0,4R)$,
\begin{equation} \label{eq2}
|q(s,x)-q(t,y)|\leq c\, \|q\|_{\infty, R}R^{-\beta}(|t-s|^{1/\alpha}+|y-x|)^{\beta}
\end{equation}
holds for $(s,x), (t,y)\in Q^{\rho}(0,x_0,R)$, where $\|q\|_{\infty, R}:=\sup\limits_{(t,y)\in [0,\gamma(4R)^2]\times\mathcal{S}}|q(t,y)|$. In
particular,
\[
|p(s,x_1,y_1)-p(t,x_2,y_2)|\ \leq  c(t\wedge
s)^{-(d+\beta)/\alpha}(|t-s|^{1/\alpha}+|x_1-x_2|+|y_1-y_2|)^{\beta}.
\]
\end{theorem}
\begin{proof} With the help of Lemmas \ref{lem1} and \ref{lem2}, we can prove \eref{eq2} in the same
way as Theorem 4.9 of \cite{bass_kumagai1}.
\end{proof}

\medskip

\section{Approximations}
In this section, we first construct the approximating Markov chains and then give a necessary condition for the weak convergence of these chains to singular stable-like processes $X$ corresponding to the Dirichlet forms $(\mathcal{E},\mathcal{F})$ in \eref{form}.

For $x\in\R^d$ and $n\in\bN$, define
\[
[x]_n=\big([nx_1]/n,\dots,[nx_d]/n\big),\; \mathcal{S}_n=\big\{[x]_n:
x\in\R^d\big\}\; \text{and}\; \mathcal{S}'_n=\big\{[x]_n:
x\in\cup^d_{i=1}\R_i\big\}.
\]
For any $x$ and $y$ in $\mathcal{S}_n$, let $C_n(x,y)$ be conductance on
$\mathcal{S}_n\times\mathcal{S}_n$ satisfying
\[
\left\{
\begin{array}{cl}
\frac{\kappa_1}{|y-x|^{1+\alpha}}\leq C_n(x,y)\leq
\frac{\kappa_2}{|y-x|^{1+\alpha}}, &\quad\text{if}~y-x\in\mathcal{S}'_n-\{0\};\\
C_n(x,y)=0, &\quad\text{otherwise},
\end{array}\right.
\]
$Y^n$ the Markov chain associated with $C_n(x,y)$ and $(\mathcal{E}^n,\mathcal{F}_n)$ the Dirichlet form corresponding to $Y^n$. Let
$p^n(t,x,y)$ be the transition density of $Y^n$. We can extend
$C_n(x,y)$ to $\R^d\times\R^d$ as follows:
\[
C_n(x,y)=C_n([x]_n,[y]_n),\quad\text{for}\; x,y\in\R^d.
\]
If $f$ is a function on $\R^d$, we define its restriction to
$\mathcal{S}_n$ by $R_n f(x)=f(x)$ for $x\in\mathcal{S}_n$. For $\lambda>0$, let
$U^{\lambda}_n$ be the $\lambda$-resolvent for $Y^n$ and $U^{\lambda}$ the $\lambda$-resolvent for $X$.

For any $f$ and $g$ in $L^2(\mathcal{S}_n)$, set $(f,g)_n=\sum_{x\in \mathcal{S}_n} f(x)g(x)\, n^{-d}$ and $\|f\|_{2,n}=\sqrt{(f,f)_n}$.

\medskip
We next prove the Markov chain approximations to singular stable-like processes $X$. The proof of the following result is similar to those in \cite{husseini_kassmann}, \cite{bass_kumagai2} and references therein.
\begin{theorem} Suppose that for each $N\geq 1$,
\begin{equation*}
C_n([x]_n,[y]_n)1_{[N^{-1},N]}(|x-y|)\,dy\,dx\rightarrow
J(x,y)1_{[N^{-1},N]}(|x-y|)\,m(dy)\, dx
\end{equation*}
weakly in the sense of measures as $n\to\infty$. Then for each $x\in\R^d$ and each $t_0>0$
the $\P^{[x]_n}$-laws of $\{Y^n_t; 0\leq t\leq t_0 \}$ converge
weakly to the $\P^x$-law of $\{X_t; 0\leq t\leq t_0 \}$ which corresponds to the Dirichlet form \eref{form}.
\end{theorem}
\begin{proof} The proof will be done in several steps.

\medskip \noindent
\textbf{Step 1} \quad We show that any subsequence $\{n_j\}$ has a further
subsequence $\{n_{j_k}\}$ such that
$\{U^{\lambda}_{n_{j_k}}R_{n_{j_k}}f\}$ converges uniformly on
compact sets whenever $f\in C_0(\R^d)$.

For each $x\in\mathcal{S}_n$, let $Q_n(x)=\prod^d_{i=1}[x_i,
x_i+1/n]$. If $f$ is a function on $\mathcal{S}_n$,  we define its
extension to $\R^d$ by $E_nf$ which is a Lipschitz-continuous
function $\R^d\to \R$ and satisfies conditions (a) $E_n f(x)=f(x)$ for $x\in\mathcal{S}_n$ and (b) $E_nf$ is linear in each $Q_n(x)$. A construction of such function $E_nf$ is available in \cite{bass_kumagai1}.

For fixed $f\in C_0(\R^d)$,
\[
||U^{\lambda}_n(R_n(f))||_{\infty}\leq ||
R_n(f)||_{\infty}/\lambda\leq || f||_{\infty}/\lambda.
\]
Therefore $\{U^{\lambda}_n(R_nf)\}$ is uniformly bounded, so is
$\{E_nU^{\lambda}_n R_nf\}$. For any $x$ and $y$ in $\R^d$,
\begin{align*}
\big|U^{\lambda}_nR_nf([y]_n)-U^{\lambda}_nR_nf([x]_n)\big|
\leq& \int^{t_0}_0e^{-\lambda t}\sum_{z\in\mathcal{S}_n}|f(z)|\big|
p^n(t,[x]_n,z)-p^n(t,[y]_n,z)\big|\,n^{-d}\,dt\\
&\quad+\int^{\infty}_{t_0}e^{-\lambda t}\sum_{z\in\mathcal{S}_n}\mid
f(z)|\big|p^n(t,[x]_n,z)-p^n(t,[y]_n,z)\big|\,n^{-d}\,dt\\
\leq &\, 2\,\| f\|_{\infty}\,t_0+c_1\,{t_0}^{-(d+\beta)/\alpha}|[x]_n-[y]_n|^{\beta},
\end{align*}
where we used Theorem
\ref{regular} in the last inequality and $c_1$ is a constant independent of $n$. For any $\epsilon>0$, we choose $t_0$ small enough such that the first term
is less than $\epsilon/3$. Fix such $t_0$, we next estimate the second term. Note that
$|[x]_n-[y]_n|\leq |x-y|+2\sqrt{d}/n$. We obtain
\[
|[x]_n-[y]_n|^{\beta}\leq c_2\big(|x-y|^{\beta}+n^{-\beta}\big).
\]
For the fixed $t_0$, there exists $n_0\in\bN$ such that
$c_1\,{t_0}^{-(d+\beta)/\alpha}\,c_2\,n^{-\beta}<\epsilon/3$ for all
$n\geq n_0$. Hence
\[
|U^{\lambda}_nR_nf([y]_n)-U^{\lambda}_nR_nf([x]_n)|\leq
\epsilon
\]
for all $n\geq n_0$ and $|y-x|\leq 1/n_0$. Since
$|[x]_n-x|\leq\sqrt{d}/n$, by the definition of $E_n$ and Theorem
\ref{regular},
\[
\big|E_nU^{\lambda}_nR_nf(x)-U^{\lambda}_n R_nf([x]_n)\big|\leq
c_3\,n^{-\beta},\quad\text{for all}\; x\in\R^d.
\]
Therefore, for any $\epsilon>0$, there exists $n_1\in\bN$ such that
\begin{equation} \label{eq3}
|E_nU^{\lambda}_nR_nf(y)-E_nU^{\lambda}_nR_nf(x)|\leq \epsilon
\end{equation}
for all $n\geq n_1$ and $|y-x|\leq 1/n_1$. This implies that $\{E_nU^{\lambda}_nR_nf\}$
is equicontinuous on $\R^d$. By the Arzel\`{a}-Ascoli Theorem,
any subsequence of $\{E_nU^{\lambda}_nR_nf\}$ has a convergent further subsequence. Therefore, any
subsequence $\{n_j\}$ has a further subsequence $\{n_{j_k}\}$ such
that $\{U^{\lambda}_{n_{j_k}}R_{n_{j_k}}f\}$ converges uniformly on
compact sets whenever $f\in C_0(\R^d)$.

\medskip \noindent
\textbf{Step 2} \quad Suppose that $\{n'\}$ is a subsequence that
$\{U^{\lambda}_{n'}R_{n'}f\}$ converges uniformly to some $H$. We show
$H\in\mathcal{F}$.

For $\lambda>0$, let $u_n=U^{\lambda}_nR_nf$. Then
\begin{equation} \label{eq4}
\mathcal{E}^n(u_n,u_n)=(R_n f,u_n)_n-\lambda\|u_n\|^2_{2,n}.
\end{equation}
Moreover,
\[
\|\lambda u_n\|_{2,n}=\|\lambda U^{\lambda}_n R_n f\|_{2,n}\leq
\|R_n f\|_{2,n}\leq \sup_n \|R_n f\|_{2,n}<\infty,
\]
where we used $\lim\limits_{n\to\infty}\|R_n f\|_{2,n}=\|f\|_2<\infty$
in the last inequality. 

Therefore the right hand side of \eref{eq4} is
bounded by
\[
|(R_nf,u_n)_n|+\lambda\|u\|^2_{2,n}\leq \frac{1}{\lambda}\|R_n
f\|_{2,n}\|\lambda u_n\|_{2,n}+\frac{1}{\lambda}\|\lambda
u_n\|^2_{2,n}\leq\frac{2}{\lambda}\sup_n\|R_nf\|^2_{2,n}.
\]
This implies that $\{\mathcal{E}^n(u_n,u_n)\}$ is uniformly bounded.

Since $u_{n'}$ converges uniformly to $H$ on $\overline{B(0,N)}$ for
$N>0$, by assumption,
\begin{align*}
& \int\int_{N^{-1}\leq |y-x|\leq
N}\big(H(y)-H(x)\big)^2J(x,y)\,m(dy)\,dx\\
&\leq  \lim\sup_{n'\to\infty}
\sum_{x,y\in\mathcal{S}_{n'},|y-x|\leq N}\big(u_{n'}(y)-u_{n'}(x)\big)^2C_{n'}(x,y)(n')^{-1-d}\\
&\leq  \limsup_{n'\to\infty}\mathcal{S}^{n'}(u_{n'},u_{n'})\\
&<\infty.
\end{align*}
On the other hand,
\[
\int_{\overline{B(0,N)}}H^2(x)\,dx\leq \frac{1}{\lambda^2}\sup_n\|R_nf\|^2_{2,n}<\infty.
\]
Combining these estimates and letting $N\to\infty$, we have
\[
\mathcal{E}(H,H)+\|H\|^2_2<\infty
\] 
and thus $H\in\mathcal{F}$.

\medskip \noindent
\textbf{Step 3} \quad  We show
$\lim\limits_{n'\to\infty}\mathcal{E}^{n'}(u_{n'},g)=\mathcal{E}(H,g)$ for all $g\in C^1_0(\R^d)$.

Since $g\in C^1_0(\R^d)$, we can choose $K$ large enough so that the
support of $g$ is contained in $B(0,K)$. By the Cauchy-Schwartz inequality,
\begin{align*}
&\Big|\sum_{x,y\in\mathcal{S}_n,|y-x|>N}(u_n(y)-u_n(x))(
g(y)-g(x)) C_n(x,y)\,n^{-1-d}\Big|\\
&\leq \big(\mathcal{E}^{n}(u_n,u_n)\big)^{1/2}\Big(\sum_{x,y\in\mathcal{S}_n,|y-x|>N}(g(y)-g(x))^2C_n(x,y)\,n^{-1-d}\Big)^{1/2}\\
&\leq 2\,\|g\|_{\infty}\big(\mathcal{E}^n(u_n,u_n)\big)^{1/2}\Big(\sum_{x\in
B(0,K)\cap\mathcal{S}_n}\sum_{|y-x|>N}
C_n(x,y)\, n^{-1-d}\Big)^{1/2}\\
&\leq c_4\|g\|_{\infty}K^dN^{-\alpha}\big(\mathcal{E}^n(u_n,u_n)\big)^{1/2}\\
&\leq c_5\,N^{-\alpha}.
\end{align*}
Similarly,
\begin{align*}
&\Big|\sum_{x,y\in\mathcal{S}_n,|y-x|<N^{-1}}(u_n(y)-u_n(x))(
g(y)-g(x))C_n(x,y)\,n^{-1-d}\Big|\\
&\leq \big(\mathcal{E}^{n}(u_n,u_n)\big)^{1/2}\Big(\sum_{x,y\in\mathcal{S}_n,|y-x|<N^{-1}}(g(y)-g(x))^2C_n(x,y)\,n^{-1-d}\Big)^{1/2}\\
&\leq \|\nabla
g\|_{\infty}\big(\mathcal{E}^n(u_n,u_n)\big)^{1/2}\Big(\sum_{x\in
B(0,K)\cap\mathcal{S}}\sum_{|y-x|<N^{-1}}|y-x|^2C_n(x,y)\,n^{-1-d}\Big)^{1/2}\\
&\leq c_6\, \|\nabla
g\|_{\infty}K^dN^{\alpha-2}(\mathcal{E}^n(u_n,u_n))^{1/2}\\
&\leq c_7\,N^{\alpha-2}.
\end{align*}

Since $H\in\mathcal{F}$, we can choose $N$ large enough such that
\[
\big|\int\int_{|y-x|\notin[N^{-1},N]}(H(y)-H(x))(g(y)-g(x))J(x,y)\, m(dy)\, dx\big|
\]
is small. Recall that $\{n'\}$ is a
subsequence of $\{n\}$ and $U^{\lambda}_{n'}R_{n'}f$ converges
uniformly to $H$ on compact sets. Therefore,
\begin{align*}
& \sum_{x,y\in\mathcal{S}_{n'},N^{-1}\leq |y-x|\leq
N}(u_{n'}(y)-u_{n'}(x))(g(y)-g(x))C_{n'}(x,y)\,(n')^{-1-d}\\
&\rightarrow\int\int_{N^{-1}\leq |y-x|\leq N}
(H(y)-H(x))(g(y)-g(x))J(x,y)\, m(dy)\, dx.
\end{align*}
Combining these estimates gives $\lim\limits_{n'\to\infty}\mathcal{E}^{n'}(u_{n'},g)=\mathcal{E}(H,g)$.

\medskip \noindent
\textbf{Step 4} \quad We show that $\mathcal{E}(H,g)=(f,g)-\lambda(H,g)$ for all
$g\in\mathcal{F}$ and $H=U^{\lambda}f$.

From the above three steps,
\[
\mathcal{E}(H,g)=\lim_{n'\to\infty}\mathcal{E}^{n'}(u_{n'},g)=\lim_{n'\to
\infty}(f,g)_{n'}-\lambda(u_{n'},g)_{n'}=(f,g)-\lambda(H,g)
\]
for all $g\in C^1_0(\R^d)$. Note that $C^1_0(\R^d)$ is dense in
$\mathcal{F}$ with respect to the norm
$(\mathcal{E}(\cdot,\cdot)+\|\cdot\|^2_2)^{1/2}$, see Theorem 3.9 in \cite{xu}.
\[
\mathcal{E}(H,g)=(f,g)-\lambda(H,g),\quad\text{for
all}\;g\in\mathcal{F}.
\]
This implies that $H$ is the $\lambda$-resolvent of $f$ for the
process corresponding to the Dirichlet form
$(\mathcal{E},\mathcal{F})$, that is, $H=U^{\lambda}f$. According to
what we have obtained so far, we know that every subsequence of
$\{U^{\lambda}_nR_nf\}$ has a convergent further subsequence with limit
$U^{\lambda}f$. Therefore, the whole sequence $U^{\lambda}_nR_nf$
converges to $U^{\lambda}f$ whenever $f\in C_0(\R^d)$, that is,
$\lim\limits_{n\to\infty}U^{\lambda}_nR_nf=U^{\lambda}f$ for $f\in
C_0(\R^d)$.

\medskip \noindent
\textbf{Step 5} \quad For each $f\in C_0(\R^d)$, we show $\lim\limits_{n\to\infty}\P^n_t
R_n f=\P_t f$.

Using the same argument as in \textbf{Step 1}, we see that any sequence of
$\{\P^{n}_tR_nf\}$ has a uniformly convergent subsequence whenever
$f\in C_0(\R^d)$. Suppose we have a subsequence $\{n'\}$ such that
$\lim\limits_{n'\to\infty}\P^{n'}_t R_n f$ exists. Note that
\[
U^{\lambda}_n R_nf=\int^{\infty}_0e^{-\lambda
t}\,\P^n_tR_nf\, dt\quad\text{and}\quad
U^{\lambda}f=\int^{\infty}_0e^{-\lambda t}\,\P_tf\,dt.
\]
Using the uniqueness of Laplace transform and the fact
$\lim\limits_{n\to\infty}U^{\lambda}_nR_nf=U^{\lambda}f$ for $f\in
C_0(\R^d)$, we obtain that the whole sequence $\{\P^{n}_t R_n
f\}$ converges to $\P_t f$ whenever $f\in C_0(\R^d)$.

\medskip \noindent
\textbf{Step 6} \quad We show the weak convergence of the $\P^{[x]_n}$-laws of
$\{Y^n_t; 0\leq t\leq t_0\}$ for each $t_0>0$.

It suffices to show the tightness of $\{Y^{n}_t; 0\leq t\leq t_0\}$
in the space $D[0,t_0]$ and that finite-dimensional distributions of
$\{Y^n_t; 0\leq t\leq t_0\}$ converge to those of $\{X_t; 0\leq
t\leq t_0\}$. Let $\tau_n$ be stopping times bounded by $t_0$ and
$\{\delta_n\}$ a sequence of positive real numbers converging to
0. Then, by Theorem \ref{exit1} and the strong Markov property,
\[
\P^{[x]_n}\big(|Y^n_{\tau_n+\delta_n}-Y^n_{\tau_n}|>a\big)
= \P^{[x]_n}\big(|Y^n_{\delta_n}-Y^n_0|>a\big)\leq \P^{[x]_n}\big(\tau_{([x]_n,A)}(Y^n)<\gamma(a,b)\big)\leq b
\]
for all $n$ large enough such that $\delta_n\leq \gamma(a,b)$.
Moreover, $[x]_n\to x$ implies the tightness of the starting
distributions and Theorem \ref{exit1} implies the tightness of
$\max_{t\in[0,t_0]}|Y^n_t-Y^n_{t-}|$ both under $\P^{[x]_n}$. By
Theorem 1 in \cite{aldous}, we have the tightness of the
$\P^{[x]_n}$-laws of $\{Y^n_t; 0\leq t\leq t_0\}$. Suppose $f\in
C_0(\R^d)$. Then, for each $t\in[0,t_0]$,
\[
\E^xf(X_s)=\P_tf=\lim_{n\to\infty}\P^{n}_tR_{n}f=\lim_{n\to\infty}\E^{[x]_{n}}R_{n}f(Y^n_t).
\]
Thus the one-dimensional distributions of $\{Y^n_t; 0\leq t\leq
t_0\}$ converge to those of $\{X_t; 0\leq t\leq t_0\}$. Similarly,
we can prove the finite-dimensional case using the Markov property and the time-homogeneity of $Y^n$ and the result in \textbf{Step 5}.
\end{proof}

\end{document}